\documentclass{amsart}
\usepackage{amsfonts,amssymb,amsmath}

\newtheorem{theorem}{Theorem}
\newtheorem{acknowledgement}[theorem]{Acknowledgement}

\newtheorem{corollary}[theorem]{Corollary}

\newtheorem{lemma}[theorem]{Lemma}

\begin{document}
\title{On idempotent stable range one matrices}
\author{Grigore C\u{a}lug\u{a}reanu, Horia F. Pop}
\thanks{Keywords: idempotent stable range 1, clean, coprime integers, $%
2\times 2$ matrix. MSC 2010 Classification: 16U99, 16U10, 15B33, 15B36,
16-04, 15-04}
\address{Babe\c{s}-Bolyai University, Cluj-Napoca, Romania}
\email{calu@math.ubbcluj.ro, hfpop@cs.ubbcluj.ro}

\begin{abstract}
We characterize the idempotent stable range one $2\times 2$ matrices over
commutative rings and in particular, the integral matrices with this
property. Several special cases and examples complete the subject.
\end{abstract}

\maketitle

\section{Introduction}

The idempotent stable range 1 for elements in a unital ring was introduced
in \cite{chen} and further studied in \cite{wan}.

\textbf{Definition}. An element $a$ of a ring $R$ is said to have \textsl{%
(left) stable range one} (\textsl{sr1}, for short) if for any $b\in R$, the
equality $Ra+Rb=R$ implies that $a+rb$ is a unit for some $r\in R$. If $r$
can be chosen to be an idempotent, we say that $a$ has \textsl{(left)
idempotent stable range one} (\textsl{isr1}, for short).

Actually, in \cite{chen}, the definition of a (right) idempotent stable
range 1 element was given requiring $a+br$ to be only left invertible, but
it was immediately proved this is equivalent with asking $a+br$ being a unit.

Since so far, left-right symmetry \emph{for elements} with sr1 is an open
question, we shall also consider it is open for isr1 elements, and discuss
in the sequel about \emph{left idempotent sr1 elements}. From the
definition, we derive directly that $a$ has \emph{left idempotent sr1} iff
for every $x,b\in R$ and $xa+b=1$, there is an idempotent $e\in R$, called 
\textsl{unitizer} (as in \cite{ca-po}), such that $a+eb$ is a unit.
Equivalently, for every $x\in R$, there is $e^{2}=e\in R$ such that $%
a+e(xa-1)$ is a unit. An element in a unital ring is called \textsl{clean}
if it is a sum of an idempotent and a unit, and \textsl{strongly clean} if
the idempotent and the unit commute.

Taking $e=0$ shows that \emph{units have }(not only sr1 but\emph{\ }also)%
\emph{\ isr1}. \emph{Zero} has trivially isr1 (take $e=1$).

Moreover, \emph{isr1 elements are clean} (just take $x=0$), but the converse
fails (see the starting example of the next section).

In Section 2, we specialize the characterization given in \cite{ca-po}, for
sr1, $2\times 2$ matrices over any commutative ring to idempotent sr1
matrices and show that for such matrices, this notion is left-right
symmetric. Next, we characterize the $2\times 2$ integral idempotent sr1
matrices together with some special cases, including idempotents, nilpotents
and matrices with zero second row. For the latter, "clean" and "idempotent
sr1" turn out to be equivalent.

The last section is dedicated to examples and comments.

For any unital ring $R$, $U(R)$ denotes the set of all the units and $%
\mathbb{M}_{2}(R)$ denotes the corresponding matrix ring. By $E_{ij}$ we
denote the square matrix having all entries zero, excepting the $(i,j)$
entry, which is $1$.

\section{$2\times 2$ idempotent stable range 1 matrices}

First notice that for any idempotent $e$, $2e$ is strongly clean: indeed, $%
2e=1+(2e-1)$ and $2e-1=(2e-1)^{-1}\in U(R)$.

Next we give an example of a \emph{strongly clean element which has not
idempotent sr1}. Namely, we show (directly from definition) that $2E_{11}$
has not idempotent sr1 in $\mathbb{M}_{2}(R)$, for any commutative ring $R$
such that $2\notin U(R)$ and $2R+1\nsubseteqq U(R)$ (e.g., $R=\mathbb{Z}$).
Suppose the contrary.

Let $a\in R$ such that $2a+1\notin U(R)$ and take $X=\left[ 
\begin{array}{cc}
a & 0 \\ 
0 & 0%
\end{array}%
\right] $. By left isr(1), there is an idempotent (unitizer) $E$, such that $%
2E_{11}+E(2XE_{11}-I_{2})\in U(\mathbb{M}_{2}(R))$. By computation

$\det (2E_{11}+E\left[ 
\begin{array}{cc}
2a-1 & 0 \\ 
0 & -1%
\end{array}%
\right] )\in U(R)$. Since the trivial idempotents $E=0_{2}$ and $E=I_{2}$
are not suitable (as for the latter, the determinant is $-(2a+1)$), we may
assume $E=\left[ 
\begin{array}{cc}
x & y \\ 
z & 1-x%
\end{array}%
\right] $ with $x(1-x)=yz$. Then $\det (2E_{11}+E(2XE_{11}-I_{2}))=$

\noindent $=\det \left[ 
\begin{array}{cc}
2+(2a-1)x & -y \\ 
(2a-1)z & x-1%
\end{array}%
\right] =2(x-1)\in U(R)$, a contradiction. Hence this is an example of \emph{%
(strongly) clean matrix that has not idempotent sr1}.

\bigskip

The important properties elements in Ring Theory may have, say, idempotent
or nilpotent or unit, each separately, \emph{is invariant under conjugations}
but not each (separately) is invariant under equivalences. Namely, \emph{%
idempotents and nilpotents are (separately) not invariant under equivalences}%
, but \emph{units}, and more generally, \emph{sr1 elements, are invariant
under equivalences}. A simple example is $E_{11}\left[ 
\begin{array}{cc}
0 & 1 \\ 
1 & 0%
\end{array}%
\right] =E_{12}$, that is, an idempotent is equivalent to a nilpotent.

Therefore, as noticed in \cite{ca-po}, for the determination of sr1 matrices
over elementary divisor rings, the \emph{diagonal reduction} is useful, but 
\emph{it is not, in the determination of} idempotents, or nilpotents or 
\emph{idempotent sr1 elements}.

In particular, the proof of the \emph{multiplicative closure for sr1 elements%
} (see \cite{che}), does specialize to idempotent sr1.

Indeed, an example showing that the set of all the idempotent sr1 elements 
\emph{is not} multiplicatively closed is given at the end of the paper.

\bigskip

First recall from \cite{ca-po}, the following characterization

\begin{theorem}
\label{lsr1}Let $R$ be a commutative ring and $A\in \mathbb{M}_{2}(R)$. Then 
$A$ has left stable range 1 iff for any $X\in \mathbb{M}_{2}(R)$ there
exists $Y\in \mathbb{M}_{2}(R)$ such that 
\begin{equation*}
\det (Y)(\det (X)\det (A)-\mathrm{Tr}(XA)+1)+\det (A(\mathrm{Tr}(XY)+1))-%
\mathrm{Tr}(A\mathrm{adj}(Y))
\end{equation*}%
is a unit of $R$. Here $\mathrm{adj}(Y)$ is the adjugate matrix.
\end{theorem}

Notice that $\det (A(\mathrm{Tr}(XY)+1))=(\mathrm{Tr}(XY)+1)^{2}\det (A)$.

We obtain a \emph{characterization for idempotent sr1 matrices} just adding
the condition $Y^{2}=Y$.

As this was done (see \cite{ca-po}) for sr1, $2\times 2$ matrices, we obtain

\begin{corollary}
Let $R$ be a commutative ring and $A\in \mathbb{M}_{2}(R)$. Then $A$ has
left idempotent stable range 1 iff $A$ has right idempotent stable range 1.
\end{corollary}

\begin{proof}
Using the properties of determinants, the properties of the trace and the
commutativity of the base ring, it is readily seen that changing $A,X,Y$
into transposes and reversing the order of the products does not change the
condition in the previous theorem.
\end{proof}

\bigskip

We have also proved (see \cite{ca-po}) that a matrix $A\in \mathbb{M}_{2}(%
\mathbb{Z})$ has sr1 iff $\det A\in \{-1,0,1\}$. Since $\det (A)\in \{\pm
1\} $ yield precisely the units, in order to determine the idempotent sr1
integral $2\times 2$ matrices, it remains to deal with (nonunits in) the
case $\det (A)=0$.

As our first main result we have the following characterization

\begin{theorem}
\label{unu}A noninvertible $2\times 2$ integral matrix $A$ has idempotent
sr1 iff $\det (A)=0$, the entries of $A$ are (setwise) coprime and there
exists a nontrivial idempotent $E$ such that $\mathrm{Tr}(AE)\in \{\pm 1\}$.
If $A=\left[ 
\begin{array}{cc}
a_{11} & a_{12} \\ 
a_{21} & a_{22}%
\end{array}%
\right] $, the last conditions are equivalent to the existence of integers $%
x,y,z$ such that $-a_{11}(1-x)+a_{12}z+a_{21}y-a_{22}x\in \{\pm 1\}$ and $%
x(1-x)=yz$.
\end{theorem}

\begin{proof}
Since sr1 integral matrices $A$ have $\det (A)\in \{-1,0,1\}$ and we have
excluded the units, $\det (A)=0$ is necessary. The condition in the previous
characterization becomes: \ for every $X=\left[ 
\begin{array}{cc}
a & b \\ 
c & d%
\end{array}%
\right] $, there is $Y=\left[ 
\begin{array}{cc}
x & y \\ 
z & t%
\end{array}%
\right] =Y^{2}$ such that 
\begin{equation*}
\det (Y)(-\mathrm{Tr}(XA)+1)-\mathrm{Tr}(A\mathrm{adj}(Y))\in \{\pm 1\}.
\end{equation*}%
Over $\mathbb{Z}$, any (idempotent) unitizer is $Y=\left[ 
\begin{array}{cc}
x & y \\ 
z & 1-x%
\end{array}%
\right] $ with $x(1-x)=yz$ or else $Y\in \{0_{2},I_{2}\}$.

The unitizer cannot be $0_{2}$ (just by replacement) and it could be $I_{2}$
whenever $1-\mathrm{Tr}(XA)-\mathrm{Tr}(A)\in \{\pm 1\}$. In the latter
case, since for a given matrix $A$, $1-\mathrm{Tr}(XA)-\mathrm{Tr}(A)\in
\{\pm 1\}$ cannot hold for all $X\in \mathbb{M}_{2}(\mathbb{Z})$, for all
the other $X$, if $A$ is indeed isr1, $Y$ must be a nontrivial idempotent.

Hence, we can assume $\det (Y)=0$ and $\mathrm{Tr}(Y)=1$, and so $A$ has
isr1 precisely when there are integers $x,y,z$ such that $-\mathrm{Tr}(A%
\mathrm{adj}(Y))\in \{\pm 1\}$, that is $E=\mathrm{adj}(Y)$. Notice that $%
\det (Y)=\det (\mathrm{adj}(Y))$ and $\mathrm{Tr}(Y)=\mathrm{Tr}(\mathrm{adj}%
(Y))$, both are idempotent or not.

Therefore (surprisingly) the unitizer is independent of $X$ and the
condition amounts to: a given matrix $A=\left[ 
\begin{array}{cc}
a_{11} & a_{12} \\ 
a_{21} & a_{22}%
\end{array}%
\right] $ has isr1 iff there exist integers $x,y,z$ such that $%
-a_{11}(1-x)+a_{12}z+a_{21}y-a_{22}x\in \{\pm 1\}$ and $x(1-x)=yz$.

Such integers (and the corresponding unitizer) exist iff the entries of $A$
are (setwise) coprime and the coefficients in the linear combination give $%
\det (Y)=0$ and $\mathrm{Tr}(Y)=1$, as stated.
\end{proof}

In the sequel, we discuss several special cases.

\begin{corollary}
A $2\times 2$ integral matrix $A$, with 3 zero entries has idempotent sr1
iff the nonzero entry is $\pm 1$.
\end{corollary}

\begin{proof}
One way, just notice that, for an integer $n$, $\{n,0,0,0\}$ are (setwise)
coprime iff $n\in \{\pm 1\}$. Conversely, it is easily checked that: $E_{22}$
is an (idempotent) unitizer for $E_{11}$ and vice-versa, $E_{11}+E_{21}$ is
an (idempotent) unitizer for $E_{12}$ and $E_{12}+E_{22}$ is an (idempotent)
unitizer for $E_{21}$.
\end{proof}

The section begun by giving a direct (from definition) proof for $%
isr1(2E_{11})\neq 1$. Now this also follows from the previous corollary.

\begin{corollary}
Idempotent $2\times 2$ integral matrices have idempotent sr1. Only the
nilpotent matrices similar to $\pm E_{12}$ have idempotent sr1.
\end{corollary}

\begin{proof}
The trivial idempotents are known to have idempotent sr1. Every nontrivial
idempotent is similar (conjugate) to $E_{11}$, so has idempotent sr1. As for
nilpotents, recall that every nilpotent matrix is similar to a multiple of $%
E_{12}$. Only those which are similar to $\pm E_{12}$ have idempotent sr1.
Indeed, let $T=\left[ 
\begin{array}{cc}
x & y \\ 
z & -x%
\end{array}%
\right] $ with $x^{2}+yz=0$ be any nilpotent matrix and let $d=\gcd (x;y)$.
Denote $x=dx_{1}$, $y=dy_{1}$ with $\gcd (x_{1};y_{1})=1$. Then $%
d^{2}x_{1}^{2}=-dy_{1}z$ and since $\gcd (x_{1};y_{1})=1$ implies $\gcd
(x_{1}^{2};y_{1})=1$, it follows $y_{1}$ divides $d$. If $d=y_{1}y_{2}$ then 
$T$ is similar to $y_{2}E_{12}$.
\end{proof}

\textbf{Examples}. 1) $2E_{12}$ is not similar to any of $\pm E_{12}$.
Indeed, if $(2E_{12})U=\pm UE_{12}$, for some $U$, then three entries of $U$
vanish, so $U$ cannot be a unit. Hence $isr(2E_{12})\neq 1$.

2) In the case of nilpotents, $\left[ 
\begin{array}{cc}
6 & 3 \\ 
-12 & -6%
\end{array}%
\right] $ is similar to $3E_{12}$ and so has \emph{not} idempotent sr1, but $%
\left[ 
\begin{array}{cc}
3 & 9 \\ 
-1 & -3%
\end{array}%
\right] $ is similar to $E_{12}$ and so has idempotent sr1.

\bigskip

In what follows we specialize our characterization Theorem \ref{unu} to the
case of matrices with zero second row.

The case we deal with are matrices of form $\left[ 
\begin{array}{cc}
a & b \\ 
0 & 0%
\end{array}%
\right] $, with nonzero coprime integers $a,b$ (the case with three zeros
was already settled). Notice that we can suppose both $a,b$ being positive.

Indeed, conjugation with $\left[ 
\begin{array}{cc}
1 & 0 \\ 
0 & -1%
\end{array}%
\right] $, transforms $\left[ 
\begin{array}{cc}
a & b \\ 
0 & 0%
\end{array}%
\right] $ into $\left[ 
\begin{array}{cc}
a & -b \\ 
0 & 0%
\end{array}%
\right] $, and, we pass from $\left[ 
\begin{array}{cc}
a & b \\ 
0 & 0%
\end{array}%
\right] $ to $\left[ 
\begin{array}{cc}
-a & b \\ 
0 & 0%
\end{array}%
\right] $ by rewriting the B\'{e}zout identity $ax+bz=1$ as $(-a)(-x)+bz=1$.

Such special matrices were studied with respect to cleanness in \cite{KL}.
Two consecutive reductions were made there: from $a<b$ to $a>b$, and then
from $a>b$ to $a\geq 2b$. We first show that these transformations can be
performed also for idempotent sr1 matrices.

\begin{lemma}
Suppose $a<b$ are coprime positive integers and $q\in \mathbb{Z}$. Then $%
\left[ 
\begin{array}{cc}
a & b \\ 
0 & 0%
\end{array}%
\right] $ has idempotent sr1 iff $\left[ 
\begin{array}{cc}
a & b-qa \\ 
0 & 0%
\end{array}%
\right] $ has idempotent sr1.
\end{lemma}

\begin{proof}
One way is obvious (take $q=0$). Conversely, suppose $ax+bz=1$. Then $%
a(x+qz)+(b-qa)z=1$ so $a,b-qa$ are also coprime. Both matrices (we denote
these $A$ and $A^{\prime }$) have zero determinant so, according to Theorem %
\ref{unu}, if there is a nontrivial idempotent with $\mathrm{Tr}(AE)\in
\{\pm 1\}$, we have to indicate a nontrivial idempotent $E^{\prime }$ such
that $\mathrm{Tr}(A^{\prime }E^{\prime })\in \{\pm 1\}$. This amounts to
complete the column $\left[ 
\begin{array}{c}
x+qz \\ 
z%
\end{array}%
\right] $, to the right, up to a nontrivial idempotent. Set $E=\left[ 
\begin{array}{cc}
x & y \\ 
z & 1-x%
\end{array}%
\right] $ with $x(1-x)=yz$. Then $E^{\prime }=\left[ 
\begin{array}{cc}
x+qz & y+q-2xq+q^{2}z \\ 
z & 1-(x+qz)%
\end{array}%
\right] $ yields $\mathrm{Tr}(A^{\prime }E^{\prime })\in \{\pm 1\}$, as
desired.
\end{proof}

As our second main result, we are now in position to prove

\begin{theorem}
\label{doi}For any coprime nonzero integers $a,b$ and $A=\left[ 
\begin{array}{cc}
a & b \\ 
0 & 0%
\end{array}%
\right] $ the following conditions are equivalent;

(i) $A$ has idempotent sr1;

(ii) $A$ is clean.
\end{theorem}

\begin{proof}
As already mentioned, we first show that there is no loss of generality in
working with coprime (positive) integers $a\geq 2b$.

Indeed, if $0<a<b$ and $b=qa+r$ is the division with quotient $q$ and
reminder $r$, we have $0<r<a$ and, using the previous lemma, this is the
passage from $a<b$ to $a>r$.

Next, suppose $b<a\leq 2b$. Then, using the previous lemma, we pass from $%
a>b $ to $a>b-a$ (here $b-a$ is negative), and finally we pass from $a>b$ to 
$a>a-b$. It just remains to notice that $2b\geq a$ is equivalent to $a\geq
2(a-b)$ and we are done.

Since clean matrices $A$ with $a\geq 2b$ were characterized in \cite{KL}, by 
$a\equiv \pm 1$ (mod $b$), it only remains to show that in this case $A$ has
also isr1. This follows from Theorem \ref{unu}: if $a+bz=\pm 1$ we take $%
\mathrm{adj}(Y)=\left[ 
\begin{array}{cc}
1 & 0 \\ 
z & 0%
\end{array}%
\right] $, that is, $Y=\left[ 
\begin{array}{cc}
0 & 0 \\ 
-z & 1%
\end{array}%
\right] $ and so $\mathrm{Tr}(A\mathrm{adj}(Y))=\pm 1$, as desired.
\end{proof}

\section{Details on unitizers}

We first recall briefly some details on the B\'{e}zout identity (over $%
\mathbb{Z}$, we discuss the solutions of a linear Diophantine equation with
coprime coefficients).

If $a,b$ are \emph{coprime positive integers} there exist integers $x_{0}$, $%
z_{0}$ such that $ax_{0}+bz_{0}=1$. The other solutions of the equation $%
ax+bz=1$ are $(x_{0}+kb,z_{0}-ka)$ for any integer $k$. Among these there
exist precisely 2 \emph{minimal} pairs $(x,z)$, such that $\left\vert
x\right\vert <b$, $\left\vert z\right\vert <a$. Clearly, for any solution, $x
$ and $z$ have opposite signs. Moreover, one minimal solution has $x<0$ and $%
z>0$ and the other has $x>0$ and $z<0$.

The next simple result will be used.

\begin{lemma}
\label{trei}Let $(x_{0},z_{0})$ be a given solution of the equation $ax+bz=1$
and $x=x_{0}+kb$, $z=z_{0}-ka$, the general solution. Then $z$ divides $x-1$
if there is an integer $l$ such that $(ak-z_{0})(al+b)=a-1$.
\end{lemma}

\begin{proof}
The divisibility amounts to $x_{0}-1+kb=(z_{0}-ka)l$ for some integer $l$.
Multiplying by $a$, decomposing and using $ax_{0}+bz_{0}=1$ gives the
equality in the statement.
\end{proof}

Since $a-1$ has finitely many divisors, the equality cannot be satisfied for
every integer $k$, so even if $z_{0}$ divides $x_{0}-1$, not all $z$ divide $%
x-1$.

Clearly, a necessary condition for this divisibility is that $al+b$ divides $%
a-1$ for some integer $l$.

Also notice that we actually consider $ax+bz=\pm 1$, so the solutions $%
(x,z)\,$of $ax+bz=1$, but also $(-x,-z)$.

\bigskip

Next, in order to avoid the double reduction in Theorem \ref{doi} and to
have some direct proofs (of isr1) for matrices with zero second column, in
the next result we provide unitizers in \emph{all} the possible cases.

\begin{theorem}
Suppose $a,b$ are coprime positive integers and let $x,z$ be any solution to 
$ax+bz=1$. The matrix $A=\left[ 
\begin{array}{cc}
a & b \\ 
0 & 0%
\end{array}%
\right] $ with $a<b$ has idempotent sr1 iff $z$ divides $1-x$ or $1+x$. If $%
a>b$, $A$ has idempotent sr1 iff $x=1$ or else $\left\vert z\right\vert
=\left\vert 1\pm x\right\vert $.
\end{theorem}

\begin{proof}
Suppose $a$ and $b$ are coprime positive integers and let $x,z$ be any
solution to $ax+bz=1$. As already seen before, for the conditions $\mathrm{Tr%
}(\mathrm{adj}(Y))=\pm 1$, we just have to complete the column $\left[ 
\begin{array}{c}
x \\ 
z%
\end{array}%
\right] $ to the right, up to a nontrivial idempotent $\left[ 
\begin{array}{cc}
x & y \\ 
z & 1-x%
\end{array}%
\right] $ together with $x(1-x)=yz$. Since $x$ and $z$ are coprime, this
amounts to some divisibilities.

\textbf{A}. Suppose $0<a<b$. Here $\left\vert x\right\vert >\left\vert
z\right\vert $, these have opposite signs and $z$ must divide $x-1$ or $x+1$
(as noticed, the solution $(-x,-z)$ is also suitable).

1) For $ax+bz=1$, suppose $z$ divides $x-1$, that is $x-1=kz$, for some
integer $k$. We take $\mathrm{adj}(Y)=\left[ 
\begin{array}{cc}
x & -kx \\ 
z & -kz%
\end{array}%
\right] $ for which $\det (\mathrm{adj}(Y))=0$ and $\mathrm{Tr}(\mathrm{adj}%
(Y))=1$ so $\mathrm{adj}(Y)$ is idempotent. So is $Y$ and since $A\mathrm{adj%
}(Y)=\left[ 
\begin{array}{cc}
1 & -k \\ 
0 & 0%
\end{array}%
\right] $, $\mathrm{Tr}(A\mathrm{adj}(Y))=1$ follows and we are done.

2) For $ax+bz=1$, suppose $z$ divides $x+1$, that is $x+1=kz$, for some
integer $k$. We take $\mathrm{adj}(Y)=\left[ 
\begin{array}{cc}
-x & kx \\ 
-z & kz%
\end{array}%
\right] $ for which $\det (\mathrm{adj}(Y))=0$ and $\mathrm{Tr}(\mathrm{adj}%
(Y))=1$ so $\mathrm{adj}(Y)$ is idempotent. So is $Y$ and since $A\mathrm{adj%
}(Y)=\left[ 
\begin{array}{cc}
-1 & k \\ 
0 & 0%
\end{array}%
\right] $, $\mathrm{Tr}(A\mathrm{adj}(Y))=-1$ follows, as desired.

\textbf{B}. If $a>b>0$ are coprime integers and $ax+bz=1$ then $\left\vert
x\right\vert <\left\vert z\right\vert $ and these have opposite signs ($%
x>0>z $ or else $x<0<z$). Here $x\in \{\pm 1\}$ or $\left\vert z\right\vert
=\left\vert 1\pm x\right\vert $.

1) If $x=1$ we have a unitizer indicated in the proof of Theorem \ref{doi}.
The case $x=-1$ is similar.

2) If $\left\vert z\right\vert =\left\vert 1\pm x\right\vert $, we find
unitizers as in the case \textbf{A} above.
\end{proof}

\bigskip

Related to the divisibilities above, we give the following

\textbf{Examples}. 1) For $(a,b)=(8,13)$ the minimal pairs are $(5,-3)$ and $%
(-8,5)$. For none $z$ divides $x-1$, but for $(5,-3)$, $-3$ divides $5+1$.

2) For $(15,23)$, the minimal pairs are $(-3,2)$ and $(20,-13)$. Here $2$
divides $-3-1$.

3) For $(5,7)$, the minimal pairs are $(3,-2)$ and $(-4,3)$. Here $-2$
divides $3-1$ and also $3$ divides $-4+1$.

4) For $(5,9)$, the minimal pairs are $(2,-1)$ and $(-7,4)$. Here $-1$
divides $2\pm 1$ and also $4$ divides $-7-1$.

\textbf{Nonexamples}. 1) For $(12,17)$ the minimal pairs are $(10,-7)$ and $%
(-7,5)$.

2) For $(12,19)$, the minimal pairs are $(8,-5)$ and $(-11,7)$.

3) For $(12,31)$, the minimal pairs are $(13,-5)$ and $(-18,7)$.

4) For $(13,18)$, the minimal pairs are $(7,-5)$ and $(-11,8)$.

5) For $(51,71)$, the minimal pairs are $(-32,23)$ and $(39,-28)$.

\bigskip

\textbf{Remark}. While for \emph{examples}, indicating a pair (minimal or
not) is sufficient in order to construct an idempotent unitizer, and so to
check the isr1 property, for the \emph{nonexamples} (as the referee pointed
out) a proof is necessary in order to show that if the conditions in the
previous Theorem are \emph{not} fulfilled for a given pair (e.g., a minimal
pair), these are \emph{not} fulfilled for any other solution of $ax+bz=1$.
This is clear in the \textbf{B} case and not hard to check in the \textbf{A}
case.

Recall that if $(x_{0},z_{0})$ is a (minimal) solution for $ax+bz=1$ then
the other solutions are given by $x=x_{0}+kb$, $z=z_{0}-ka$. Of course, it
would suffice to show that if $z_{0}$ does not divide $x_{0}-1$ then every $%
z $ does not divide the corresponding $x-1$.

Unfortunately this is not true in general as shows the following

\textbf{Example}. For $(a,b)=(2,5)$ a solution pair is $(-7,3)$ for which $3$
does not divide $-8=-7-1$. The general solution $x=-7+5k$, $z=3-2k$ gives
for $k=1$ the (minimal) solution $(-2,1)$ for which $1$ divides $-3=-2-1$.

Therefore, another type of verification is necessary for the nonexamples,
starting with a given pair (e.g., a minimal pair). The implication "$z_{0}$
divides $x_{0}-1$ then $z$ divides $x-1$" must be verified for each
nonexample, separately, and this amounts to solve a quadratic Diophantine
equation ! As seen in Lemma \ref{trei}, this equation is 
\begin{equation*}
(ak-z_{0})(al+b)=a-1
\end{equation*}%
with unknowns $k$, $l$. Especially when $a-1$ is a prime, it sometimes
suffices to check that $al+b$ does not divide $a-1$. This can be easily done
for examples (1)-(3).

\bigskip

Next, we reconsider example 4, from the nonexamples list above, and give all
details.

For $13x+18z=1$ the minimal pairs are $(x_{0},z_{0})\in \{(7,-5),(-11,8)\}$,
for which clearly $z_{0}$ does not divide $x_{0}\pm 1$.

It suffices to discuss one minimal pair, say the first pair, for which the
general solution is $x=7+18k$, $z=-5-13k$. Would $z$ divide $x-1$, then $%
6+18k=-(5+13k)l$ for some integer $l$.

As seen above, this quadratic Diophantine equation in the unknowns $k$, $l$
can be written $(13k+5)(13l+18)=12$. That this equation has no solutions can
be shown browsing the pairs of two integers whose product is $12$, or using
any software online (e.g., \cite{mat}).

If $z$ would divide $x+1$, similarly we reach the equation $8+18k=-(5+13k)l$%
, also with no solutions.

\bigskip

Notice that we used the minimal pairs just to keep the numbers low.

\bigskip

As applications of our results, here are some more

\textbf{Examples}. 1) We can show that $A=\left[ 
\begin{array}{cc}
5 & 12 \\ 
0 & 0%
\end{array}%
\right] $ \emph{has} idempotent sr1, in two different ways.

(i) The minimal solutions $(x,z)$ of $ax+bz=1$ are $(5,-2)$, $(-7,3)$ and $%
-2 $ divides $5-1$. Moreover $3$ divides $-7+1$. Therefore we indicate a
(nontrivial idempotent) unitizer as in \textbf{A}, above: $Y=\left[ 
\begin{array}{cc}
-4 & -10 \\ 
2 & 5%
\end{array}%
\right] $.

(ii) As described in the previous section, we pass ($12=2\cdot 5+2$) from $%
(5,12)$ to $(5,2)$ for which $5\equiv 1$ (mod $2$), so $A$ is clean and
idempotent sr1 (by Theorem \ref{doi}). Actually $\left[ 
\begin{array}{cc}
5 & 12 \\ 
0 & 0%
\end{array}%
\right] =\left[ 
\begin{array}{cc}
-4 & -10 \\ 
2 & 5%
\end{array}%
\right] +\left[ 
\begin{array}{cc}
9 & 22 \\ 
-2 & -5%
\end{array}%
\right] $ is its (uniquely) clean decomposition.

2) However, we can show that $B=\left[ 
\begin{array}{cc}
12 & 5 \\ 
0 & 0%
\end{array}%
\right] $ \emph{has not} idempotent sr1 in three different ways.

(i) Since $12>5$ and $12\geq 2\cdot 5$ we can use Theorem \ref{doi}: $%
12\not\equiv \pm 1$ (mod $5$) so $B$ is not clean (and so nor isr1).

(ii) For $X=0_{2}$, if a unitizer $Y$ exists, we would have $A-Y\in U(%
\mathbb{M}_{2}(\mathbb{Z}))$, that is, $A$ would be clean. We can show that
this fails (e.g. see \cite{cal}) solving the Diophantine equations $%
5x^{2}-12xy+5x\mp y=0$ together with $12x+5z=\pm 1$ (the only solutions are $%
(0,0),(-1,0)$, none verifies both equations), or else, using \cite{KL},
where this matrix and others are given examples of (unit-regular) matrices
which are not clean.

(iii) We are in the \textbf{B} case above. The minimal pairs for $(12,5)$
are $(x,z)\in \{(3,-7),(-2,5)\}$. For none $x\in \{\pm 1\}$ nor $\left\vert
z\right\vert =\left\vert 1\pm x\right\vert $, and this is sufficient
according to the previous remark.

\bigskip

\textbf{Remark}. The case when the first row is zero (or some column is
zero) reduces to the previous discussed case.

By conjugation with $U=E_{12}+E_{21}$, we check that $A=\left[ 
\begin{array}{cc}
a & b \\ 
0 & 0%
\end{array}%
\right] $ is similar to $A^{\prime }=\left[ 
\begin{array}{cc}
0 & 0 \\ 
b & a%
\end{array}%
\right] $, and $B=\left[ 
\begin{array}{cc}
b & a \\ 
0 & 0%
\end{array}%
\right] $ is similar to $B^{\prime }=\left[ 
\begin{array}{cc}
0 & 0 \\ 
a & b%
\end{array}%
\right] $. Both $A$, $A^{\prime }$ may have isr1 but $B$, $B^{\prime }$ may
not have isr1 (see previous example: $\left[ 
\begin{array}{cc}
5 & 12 \\ 
0 & 0%
\end{array}%
\right] $ and $\left[ 
\begin{array}{cc}
12 & 5 \\ 
0 & 0%
\end{array}%
\right] $). As for zero columns we just use the transpose.

\bigskip

In closing, we provide an example of two idempotent sr1 integral matrices
whose product has \emph{not} idempotent sr1.

\textbf{Example}. Take $A=\left[ 
\begin{array}{cc}
2 & 1 \\ 
0 & 0%
\end{array}%
\right] $. Since (by Theorem \ref{unu}), $\mathrm{adj}(Y)=\left[ 
\begin{array}{cc}
1 & 0 \\ 
-1 & 0%
\end{array}%
\right] $ gives a suitable (idempotent) unitizer, $A$ has isr1. However, $%
A^{2}=\left[ 
\begin{array}{cc}
4 & 2 \\ 
0 & 0%
\end{array}%
\right] $ has not idempotent sr1, because its entries are not (setwise)
coprime.

\begin{acknowledgement}
Thanks are due to the referee for careful reading and valuable suggestions
that improved our paper.
\end{acknowledgement}


\bigskip

\begin{thebibliography}{9}
\bibitem{cal} G. C\u{a}lug\u{a}reanu \textsl{Clean integral }$2\times 2$%
\textsl{\ matrices}. Studia Sci. Math. Hungarica \textbf{55} (1) (2018),
41-52.

\bibitem{ca-po} G. C\u{a}lug\u{a}reanu, H. F. Pop \textsl{On stable range
one matrices}. to appear in Bull. Math. Soc. Sci. Math. Roumanie 2022,
https://arxiv.org/abs/2012.13909

\bibitem{chen} H. Chen \textsl{Rings with many idempotents}. Internat. J.
Math. \& Math. Sci. \textbf{22} (3) (1999), 547-558.

\bibitem{che} H. Chen, W. K. Nicholson \textsl{Stable modules and a theorem
of Camillo and Yu}. J. Pure \& Applied Alg. \textbf{218} (2014), 1431-1442.

\bibitem{KL} D. Khurana, T. Y. Lam \textsl{Clean matrices and unit-regular
matrices}. J. Algebra \textbf{280} (2) (2004), 683-698.

\bibitem{mat} K. Matthews
http://www.numbertheory.org/php/generalquadratic.html

\bibitem{wan} Z. Wang, J. Chen, D. Khurana, T.Y. Lam \textsl{Rings of
Idempotent Stable Range One}. Algebra Represent. Theory \textbf{15} (2012),
195-200.
\end{thebibliography}
\end{document}